\documentclass{article}
\usepackage{amsmath}
\usepackage{amssymb, amscd, amsthm, amsfonts, latexsym}
\usepackage{enumerate}
\usepackage[dvips]{graphicx,color}
\usepackage{cases}

\newtheorem{thm}{Theorem}[section]
\newtheorem{theorem}[thm]{Theorem}
\newtheorem{prop}[thm]{Proposition}
\newtheorem{cor}[thm]{Corollary}
\newtheorem{rem}[thm]{Remark}
\newtheorem{defi}[thm]{Definition}
\newtheorem{ex}[thm]{Example}
\newtheorem{lem}[thm]{Lemma}

\numberwithin{equation}{section}

\newcommand{\no}{\noindent}

\newcommand{\R}{\mathbb{R}}

\newcommand{\cE}{\mathcal{E}}

\newcommand{\ds}{\displaystyle}

\newcommand{\sm}{\setminus}
\newcommand{\pd}{\partial}
\newcommand{\al}{\alpha}

\newcommand{\Ga}{\Gamma}
\newcommand{\de}{\delta}
\newcommand{\De}{\Delta}
\newcommand{\ep}{\varepsilon}
\newcommand{\la}{\lambda}
\newcommand{\f}{\varphi}
\newcommand{\ome}{\omega}
\newcommand{\Ome}{\Omega}
\renewcommand{\d}{\mathrm{d}}

\renewcommand{\(}{\left(}
\renewcommand{\)}{\right)}
\renewcommand{\lvert}{\left\vert}
\renewcommand{\rvert}{\right\vert}
\renewcommand{\hat}{\widehat}
\renewcommand{\tilde}{\widetilde}

\DeclareMathOperator*{\esssup}{ess\,sup}

\DeclareMathOperator{\interi}{int}
\DeclareMathOperator{\cl}{cl}

\usepackage{xcolor}

\begin{document}

\title{\bf Strict power concavity of a convolution}

\author{\bf Jun O'Hara and Shigehiro Sakata}

\date{\today}

\maketitle

\begin{abstract} 
We give a sufficient condition for the strict parabolic power concavity of the convolution in space variable of a function defined on $\R^n \times (0,+\infty)$ and a function defined on $\R^n$.
Since the strict parabolic power concavity of a function defined on $\R^n \times (0,+\infty)$ naturally implies the strict power concavity of a function defined on $\R^n$, our sufficient condition implies the strict power concavity of the convolution of two functions defined on $\R^n$.
As applications, we show the strict parabolic power concavity and strict power concavity in space variable of the Gauss--Weierstass integral and the Poisson integral for the upper half-space.
\\

\no{\it Keywords and phrases}.
Strict power concavity, the Borell--Brascamp--Lieb inequality, strict parabolic power concavity.\\
\no 2020 {\it Mathematics Subject Classification}: 26B25, 26D15, 90C25, 52A41, 52A40.
\end{abstract}

\maketitle

\section{Introduction}

In this paper, we are interested in the strict power concavity of the convolution,
\begin{equation}\label{f*g}
f \ast g(x)
=\int_{\R^n} f(x-y) g(y) \, \d y ,\
x \in \R^n ,
\end{equation}
of two non-negative measurable functions $f$ and $g$ defined on $\R^n$.

Let us recall the notion of power concavity and its basic properties.
Let $A$ be a convex subset of $\R^n$, $f$ a non-negative function defined on $A$, and $p \in \R \cup \{ \pm \infty \}$. 
$f$ is said to be {\it $p$-concave} on $A$ if, for any $x_0, x_1 \in A$ and $\la \in [0,1]$, the inequality 
\begin{equation}\label{p-concavity}
f \( (1-\la ) x_0 + \la x_1 \) 
\geq M_p \( f \( x_0 \) , f \( x_1 \) ;\la \) 
\end{equation}
holds. 
Here, for $a , b \in [0,+\infty )$ and $\la \in [0,1]$, 
\begin{equation}\label{pth-mean}
M_p \( a,b ;\la \) :=
\begin{cases}
0 &( ab=0) ,\\
\( \( 1- \la \) a^p + \la b^p \)^{1/p} &\(ab >0 ,\ p \notin \{ \pm \infty ,0 \} \) ,\\
\max \{ a , b \} &( ab >0,\ p =+\infty) ,\\
a^{1-\la} b^\la &( ab >0,\ p =0 ),\\
\min \{ a, b\} &( ab >0 ,\ p =-\infty ) 
\end{cases}
\end{equation}
is called the {\it $p$-th mean} of $a$ and $b$ of ratio $\la$.
$f$ is said to be {\it strictly $p$-concave} on $A$ if the inequality \eqref{p-concavity} strictly holds for any distinct $x_0, x_1 \in A$ and $\la \in (0,1)$.

When $f$ is positive on $A$ and $p \in \R$, $f$ is $p$-concave if and only if $x \mapsto f (x)^p$ is concave for $p \in (0,+\infty)$, $x \mapsto  \log f(x)$ is concave for $p =0$, and $x \mapsto f(x)^p$ is convex for $p \in (-\infty , 0)$ (see Subsect. 2.2 for the details).
As a consequence of Jensen's inequality, if $p \geq q$, then, for any $a, b \in [0,+\infty)$ and $\la \in [0,1]$,
\begin{equation}\label{M_monotonicity}
M_p (a,b; \la ) \geq M_q (a,b ; \la)
\end{equation}
holds (see, for example, \cite[Sect. 2.9]{HLP}).
Thus, for any $p \in \R \cup \{ + \infty\}$, $p$-concave (resp. strictly $p$-concave) functions are $-\infty$-concave (resp. strictly $-\infty$-concave).
 
$-\infty$-concavity is also called {\it quasi-concavity}, and we use this terminology hereafter.
It directly follows from definition that any strictly quasi-concave function has at most one global maximum point. 
Furthermore, if $f$ is strictly quasi-concave on $A$, then, for any convex subset $C$ of $A$, the restriction of $f$ to $C$ is strictly quasi-concave on $C$.
Thanks to these properties, strict quasi-concavity plays an important role for optimization problems in, for example, economics.
Namely, for a maximization problem with an objective function $f$, if $f$ is strictly quasi-concave, then we have at most one global optimal solution.

As Gardner explains in \cite[Sect. 11]{Gearly}, the power concavity of a convolution can be derived from the {\it Borell--Brascamp--Lieb inequality} (BBL-inequality, for short). 
The BBL-inequality is an integral inequality (see Theorem \ref{BL} of this paper for the precise statement).
It was shown by Borell \cite[Theorem 3.1]{B75} and by Brascamp and Lieb \cite[Theorem 3.3]{BL}, independently, around the same time. 
The proof of the BBL-inequality can be found in, for example, \cite[Sect. 3.3]{DJD}, \cite{Gup} and \cite{Rino}.
These references include probabilistic applications of the BBL-inequality.

Let us review the process of deriving the power concavity of \eqref{f*g} from the BBL-inequality according to \cite[Sect. 11]{Gearly} (see also \cite[Sect. 3.3]{DJD} and \cite[Sect. 2]{U}).
H\"{o}lder's inequality implies that, for $a, b, c, d \in [0,+\infty)$, $p, q \in \R \cup \{ \pm \infty \}$ and $\la \in [0,1]$, if $p + q \geq 0$, then 
\begin{equation}\label{Lemma10.1}
M_p ( a,b; \la ) M_q (c,d ;\la ) \geq M_\ell (ac, bd ; \la )
\end{equation}
holds, where
\begin{equation}\label{Lemma10.1ell}
\ell  =
\begin{cases}
\ds \frac{pq}{p+q} & (  p+ q \neq 0 ) ,\\
-\infty                 &\( p +q = 0 ,\  (p,q) \neq (0,0)   \),  \\  
0                  & \( (p,q) = (0,0) \) ,
\end{cases}
\end{equation}
and we understand $+ \infty + (-\infty ) = -\infty + \infty =0$ (see, for example, \cite[Lemma 10.1]{Gearly}). 
It follows from \eqref{Lemma10.1} that if $p+q \geq 0$, then, for any $p$-concave function $f$ and $q$-concave function $g$, the function 
\begin{equation}\label{F}
\R^n \times \R^n \ni (x,y ) \mapsto f(x-y) g(y) 
\end{equation}
is $\ell$-concave on $\R^n \times \R^n$.
It follows from the BBL-inequality that if $\ell \geq -1/n$, then, for any $\ell$-concave function $F$ defined on $\R^m \times \R^n$ such that the integral
\begin{equation}\label{G}
G(x) = \int_{\R^n} F(x,y) \, \d y 
\end{equation}
exists for each $x \in \R^m$, the function $G$ is $\ell / (1+n \ell )$-concave on $\R^m$ (see \cite[Theorem 4.3]{B75} and \cite[Corollary 3.5]{BL}).
Here, we understand that $\ell / (1+n \ell )$ is equal to $-\infty$ when $\ell =-1/n$ and to $1/n$ when $\ell = +\infty$.
Using this application of the BBL-inequality with $m=n$ and $F$ in \eqref{F}, we obtain the $\ell / (1+n \ell )$-concavity of \eqref{f*g}.

One of our results of this paper (Theorem \ref{thm_pconc}) is the strict version of the above.
We show that if the following conditions are satisfied, then \eqref{f*g} is strictly $\ell / (1+ n \ell )$-concave on $\R^n$:
\begin{enumerate}[(i)]
\item $f$ is strictly $p$-concave on $\R^n$.

\item $g$ is $q$-concave on $\R^n$.

\item $\R^n \sm g^{-1} (0)$ is bounded, and its interior is not empty.

\item $p +q \geq 0$ and $\ell \geq -1/n$.
\end{enumerate}
Compared to the process of deriving (not necessarily strict) power concavity, to show the {\sl strict} power concavity of \eqref{f*g}, it is essentially sufficient to add two extra assumptions, the strictness of the power concavity of $f$ and the boundedness of the support of $g$.

Our result can be applied to the Gauss--Weierstrass integral,
\begin{equation}\label{Weierstrass}
\begin{array}{rcl}
W g (x,t) &=&\displaystyle  \frac{1}{( 4 \pi t )^{n/2}} \exp \( - \frac{\vert \cdot \vert^2}{4t} \) \ast g (x)  \\[4mm]
&=&\displaystyle  \frac{1}{( 4 \pi t )^{n/2}} \int_{\R^n} \exp \( - \frac{\vert x-y \vert^2}{4t} \)  g(y)  \, \d y ,
\end{array}
\end{equation}
where $g$ is a bounded measurable function defined on $\R^n$. 
It is well-known that $Wg$ satisfies the Cauchy problem for the heat equation 
\begin{equation}
\begin{cases}
\ds \( \frac{\pd}{\pd t} -\De \) Wg (x,t) =0 , &(x,t) \in \R^n \times (0,+\infty ) ,\\
\ds Wg \( x, 0^+ \) = g(x) ,                         &x \in \R^n .
\end{cases}
\end{equation}
Since the Gauss--Weierstrass kernel is strictly $0$-concave on $\R^n$ at any fixed $t \in (0,+\infty)$, our result implies that, for any $0$-concave function $g$ such that its support is a convex body (compact convex set with non-empty interior), the function $Wg (\cdot ,t) : \R^n \to (0,+\infty)$ is strictly $0$-concave on $\R^n$ at any fixed $t \in (0,+\infty)$ (Proposition \ref{pconc_W}). 
We refer to Brascamp and Lieb's investigation \cite[Sect. 4]{BL} for the pioneering work on concavity properties of a solution of a partial differential equation (see also, for example, \cite{IS14Ann}--\cite{Kore}).

Our result can also be applied to the Poisson integral for the upper half-space, 
\begin{equation}\label{Poisson}
\begin{array}{rcl}
P g (x,t) &=&\displaystyle  \frac{2t}{\sigma_n \( S^n \)}  \( \lvert \cdot \rvert^2 +t^2 \)^{-(n+1)/2}  \ast g (x) \\[4mm]
&=&\displaystyle \frac{2t}{\sigma_n \( S^n \)} \int_{\R^n} \( \lvert x-y \rvert^2 +t^2 \)^{-(n+1)/2}  g (y) \, \d y ,
\end{array}
\end{equation} 
where $S^n$ denotes the $n$-dimensional unit sphere, $\sigma_n$ denotes the $n$-dimensional spherical Lebesgue measure, and $g$ is a bounded measurable function defined on $\R^n$.
As an analytic property, $Pg$ satisfies the Cauchy problem for the $1/2$-heat equation 
\begin{equation}
\begin{cases}
\ds \( \frac{\pd}{\pd t} + \sqrt{-\De} \) Pg(x,t) =0 , &(x,t) \in \R^n \times (0,+\infty) ,\\
\ds Pg \( x,0^+ \) = g(x) ,                                    &x \in \R^n ,
\end{cases}
\end{equation}
which is equivalent to the boundary value problem for the Laplace equation 
\begin{equation}
\begin{cases}
\ds \(  \De + \frac{\pd^2}{\pd t^2} \) Pg(x,t) =0 , &(x,t) \in \R^n \times (0,+\infty) ,\\
\ds Pg \( x,  0^+ \) = g(x) ,                                    &x \in \R^n .
\end{cases}
\end{equation}
As a geometric property, when $g$ is the characteristic function $\chi_D$ of a body (the closure of a bounded open set) $D$ in $\R^n$, $P \chi_D (x,t)$ is proportional to the solid angle of $D$ at $(x,t)$.
Namely, 
\begin{equation}
P\chi_D(x,t)= \frac{2 \sigma_n \(  \( D \ast (x,t) \) \cap \( S^n + (x,t) \) \)}{\sigma_n \( S^n \)} ,
\end{equation}
where $D \ast (x,t)$ denotes the cone of base $D$ and vertex $(x,t)$ (see \cite[p. 2157]{SakJGEA}).

It was shown in \cite[Proposition 3.7 (1)]{SakJG} that if $\Ome$ is a convex body in $\R^n$, then $P\chi_\Ome (\cdot ,t) : \R^n \to (0,+\infty)$ is strictly $-1$-concave on $\R^n$ at any fixed $t \in (0,+\infty)$.
This fact is generalized by our result since the Poisson kernel is strictly $-1/(n+1)$-concave on $\R^n$ at any fixed $t \in (0,+\infty)$.
Namely, if $q \geq1$, then, for any $q$-concave function $g$ such that its support is a convex body, the function $Pg (\cdot ,t) : \R^n  \to (0,+\infty)$ is strictly $q/(1-q)$-concave on $\R^n$ (Proposition \ref{pconc_P}). 
We remark that the characteristic function of a convex body is $+\infty$-concave on $\R^n$ and $q/(1-q) = -1$ for $q=+\infty$. 

Since $Wg$ and $Pg$ are solutions of evolution equations, it is natural to investigate those concavity properties involving the space and the time variables jointly. 
In order to investigate such a concavity property, the notion of {\it parabolic power concavity} of a function defined on a {\it parabolically convex set} in $\R^n \times (0,+\infty)$ was introduced by Ishige and Salani \cite{IS14Ann}, and exciting concavity properties of solutions of parabolic problems were shown.
As in the case of a strictly power concave function on a convex set, the strict {\sl parabolic} power concavity of a function on a {\sl parabolically} convex set guarantees the uniqueness of a global maximum point (see Subsections 2.3 and 2.4 for the precise definitions).
This is the reason why we are also interested in the strict parabolic power concavity of the convolution in space variable, 
\begin{equation}\label{Gamma}
\Ga (x,t)  = \f (\cdot ,t ) \ast \psi (x) = \int_{\R^n} \f (x-y, t) \psi (y) \, \d y ,\
 (x,t) \in \R^n \times (0,+\infty ) ,
\end{equation}
of two measurable functions $\f : \R^n \times (0,+\infty ) \to [ 0,+\infty )$ and $\psi : \R^n \to [0,+\infty)$.

The argument to show the power concavity of \eqref{f*g} also works for the parabolic power concavity of $\Ga$ in \eqref{Gamma}, that is, it is derived from the parabolic power concavity of $\f$ and the power concavity of $\psi$ through the BBL-inequality. 
In the main theorem (Theorem \ref{thm_alpconc}), we give a sufficient condition for the strict parabolic power concavity of $\Ga$ in \eqref{Gamma}.
Compared to the process of deriving (not necessarily strict) parabolic power concavity, to show the {\sl strict} parabolic power concavity of $\Ga$ in \eqref{Gamma}, it is essentially sufficient to add two extra assumptions, the ``{\sl almost-strictness}'' of the parabolic power concavity of $\f$ and the boundedness of the support of $\psi$.
Note that the strict parabolic power concavity of $\Ga$ in \eqref{Gamma} implies the strict power concavity of \eqref{f*g}.
To be precise, if $\Ga$ in \eqref{Gamma} is parabolically $p$-concave (resp. strictly parabolically $p$-concave), then, at any fixed $t$, the function $\Ga (\cdot ,t)  : \R^n \to [0,+\infty)$ is $p$-concave (resp. strictly $p$-concave).
Therefore, the strict parabolic power concavity of $\Ga$ in \eqref{Gamma} is the most important subject in this paper.

The strict parabolic power concavity of $W \chi_\Ome$ and $P \chi_\Ome$ can be derived from our main theorem, where $\Ome$ is a bounded convex subset of $\R^n$ with non-empty interior.
In particular, the strict parabolic power concavity of $P \chi_\Ome$ is the usual strict quasi-concavity on $\R^n \times (0,+\infty)$.
Recalling the application to optimization problems and the geometric interpretation of $P\chi_\Ome$, our result states that, at an art museum, when we look at a convex picture $\Ome$ on the wall from an area $E \subset \R^2 \times (0,+\infty)$, if $E$ is compact and convex, then there is a unique point with the maximum viewing solid angle.
Thus, we obtain the uniqueness of an optimal solution to a generalization of Regiomontanus' angle maximization problem.
We refer to \cite[Sect. 3.1]{N} for this kind of issue.

\section{Preliminaries}

In this section, after setting our notation, we introduce the notions of power concave functions, parabolically convex sets and parabolically power concave functions.
We also show some of their fundamental properties.

\subsection{Notation}

For a subset $X$ of $\R^n$, we denote by $\interi X$, $\cl X$ and $\chi_X$ the interior, closure and characteristic function of $X$, respectively.
For $x \in \R^n$ and $\rho \in (0,+\infty)$, we denote by $B(x , \rho)$ the open ball centered at $x$ of radius $\rho$. 
Let $S^{n-1}$ be the boundary of $B(0,1)$.
For $\mu , \nu \in \R$, $X$ and $Y \subset \R^n$, we use the Minkowski addition
\begin{equation}\label{Minkowskiadd}
\mu X + \nu Y = \{ \mu x +\nu y \,  \vert   \, x \in X , \ y \in Y \} .
\end{equation}
In particular, when $Y$ is a singleton $\{ y \}$ in \eqref{Minkowskiadd}, we write
\begin{equation}
\mu X + \nu y = \mu X + \nu \{ y \} =    \{ \mu x +  \nu y \,   \vert   \, x \in X \}.
\end{equation}
For $p \in \R \cup \{ \pm \infty \}$, $a, b \in [ 0, +\infty )$ and $\la \in [0,1]$, $M_p (a, b ; \la)$ is defined in \eqref{pth-mean}.
The {\it convex combination} of two points $x_0$ and $x_1$ of ratio $\la \in [0,1]$ is denoted by 
\begin{equation}
x_\la = (1-\la ) x_0 + \la x_1 .
\end{equation}
A compact convex set with non-empty interior in $\R^n$ is called a {\it convex body}.
For a convex body $K$ in $\R^n$, we denote by $h_K$ the {\it support function} of $K$, that is,
\begin{equation}
h_K (u) = \max \{  x \cdot u \,    \vert   \, x \in K \}   ,\ u \in S^{n-1} .
\end{equation}
Put
\begin{equation}
H^- (h , u)  = \left. \left\{ y \in \R^n   \,   \rvert   \,   y \cdot u \leq h \right\}  ,\ ( h ,u) \in \R \times S^{n-1} . 
\end{equation}

\subsection{Power concave functions}

Let us recall the definition of power concavity of a function.

\begin{defi}\label{def_pconc}
Let $A$ be a convex set in $\R^n$, $f$ a non-negative function defined on $A$, and $p \in \R \cup \{ \pm \infty \}$. 
$f$ is said to be {\it $p$-concave} on $A$ if, for any $x_0, x_1 \in A$ and $\la \in [0,1]$, the inequality 
\begin{equation}
\label{ineq-pconc}
f \( x_\la  \) \geq M_p \( f \( x_0  \) , f \( x_1  \) ; \la \)
\end{equation}
holds.
$0$-concavity and $-\infty$-concavity are also called {\it log-concavity} and {\it quasi-concavity}, respectively.
$f$ is said to be {\it strictly $p$-concave} on $A$ if both of the following conditions hold:
\begin{enumerate}[(i)]
\item $f$ is $p$-concave on $A$.

\item Equality in \eqref{ineq-pconc} holds if and only if any of the conditions $x_0 = x_1$, $\la =0$ or $\la =1$ holds.
\end{enumerate} 
\end{defi}

For $p \in \R \cup \{ \pm \infty \}$, the compositions of a $p$-concave function with a homothety and with a translation are $p$-concave.
To be precise:

\begin{rem}\label{composition}
Let $A$, $f$ and $p$ be as in Definition \ref{def_pconc}.
Suppose that $f$ is $p$-concave (resp. strictly $p$-concave) on $A$.
Then, the following statements hold:
\begin{enumerate}[(1)]
\item Let $s \in \R \sm \{ 0 \}$.
The function $(1/s) A \ni x \mapsto f (sx)$ is $p$-concave (resp. strictly $p$-concave) on $(1/s) A$.

\item Let $\xi \in \R^n$.
The function $A-\xi  \ni x \mapsto f ( x+ \xi )$ is $p$-concave (resp. strictly $p$-concave) on $A - \xi$.
\end{enumerate}
\end{rem}

It directly follows from definition that:

\begin{rem}[{\rm \cite[p. 373]{BL}}]
\label{super-level}
Let $A$ and $f$ be as in Definition \ref{def_pconc}.
$f$ is quasi-concave on $A$ if and only if, for any $a \in [0,+\infty)$, the super-level set $\{ x \in A \,\vert\, f(x) > a\}$ is convex (or empty).
\end{rem}

As mentioned in Introduction, for any $p \in \R \cup \{ \pm \infty \}$, $p$-concave functions are quasi-concave.
Thus, Remark \ref{super-level} implies:

\begin{rem}\label{supp_convex}
Let $A$, $f$ and $p$ be as in Definition \ref{def_pconc}.
If $f$ is $p$-concave on $A$, then, for any $a \in [0,+\infty )$, the super-level set $\{ x \in A \,\vert\, f(x) > a\}$ is convex (or empty).
In particular, if $f$ is $p$-concave on $A$, then $A \sm f^{-1} (0)$ is convex (or empty).
\end{rem}

It follows from the definition of $M_{+\infty}$ that:

\begin{rem}[{\rm \cite[p. 373]{BL}}]
\label{infty-concave}
Let $A$ and $f$ be as in Definition \ref{def_pconc}.
Then, the following statements hold:
\begin{enumerate}[(1)]
\item Let $\Ome \subset A$ be a convex set, and $c \in [0, +\infty)$.
The function $c \chi_\Ome$ is $+\infty$-concave on $A$.

\item Let $\Ome = A \sm f^{-1} (0)$, and fix an arbitrary $\xi \in \Ome$.
If $f$ is $+\infty$-concave on $A$, then $f = f(\xi ) \chi_\Ome$.
\end{enumerate}
\end{rem}

By definition, positive power concave functions are described as follows. 

\begin{rem}
Let $A$ and $f$ be as in Definition \ref{def_pconc}.
Suppose that $f$ is positive on $A$.
Then, the following statements hold:
\begin{enumerate}[(1)]
\item Let $p \in (0,+ \infty )$.
$f$ is $p$-concave if and only if $f^p : A \ni x \mapsto f (x)^p \in (0,+\infty)$ is concave.

\item $f$ is log-concave if and only if $\log f : A \ni x \mapsto  \log f (x) \in (0,+\infty)$ is concave.

\item Let $p \in (-\infty , 0 )$.
$f$ is $p$-concave if and only if $f^p : A \ni x \mapsto f (x)^p \in (0,+\infty)$ is convex.
\end{enumerate}
\end{rem}

Positivity and continuity are fundamental properties of strictly power concave functions. 

\begin{lem}\label{continuity}
Let $A$, $f$ and $p$ be as in Definition \ref{def_pconc}.
Suppose that $f$ is $p$-concave on $A$.
Then, the following statements hold:
\begin{enumerate}[{\rm (1)}]
\item Let $p \in \R \cup \{ +\infty \}$.
If $f$ is positive on $\interi A$, then $f$ is continuous on $\interi A$.

\item Let $p \in \R \cup \{ -\infty \}$.
If $f$ is strictly $p$-concave on $\interi A$, then $f$ has to be positive on $\interi A$.

\item Let $p \in \R$.
If $f$ is strictly $p$-concave on $\interi A$, then $f$ is continuous on $\interi A$.
\end{enumerate}
\end{lem}

\begin{proof}
(1) The statement for $p = + \infty$ follows from Remark \ref{infty-concave}. 
Since $f = (  f^p  )^{1/p}$ for $p \in \R \sm \{ 0 \}$ and $f = \exp \log f$, the statement for $p \in \R$ follows from the well-known theorem in convex analysis: any concave/convex function is continuous on the interior of its domain (see \cite[Theorem 1.5.3]{Sch} or \cite[Theorem 10.1]{Rock}).

(2)
Fix an arbitrary $x \in \interi A$. 
We take a small enough $\ep \in (0,+\infty )$ such that $B(x,\ep) \subset A$.
Let $v \in S^{n-1}$. 
Then, we have
\[
f (x) 
= f \( \frac{1}{2}  \(x+ \frac{\ep}{2} v \) +  \frac{1}{2} \( x- \frac{\ep}{2} v \)   \)  
> M_p \( f \( x+ \frac{\ep}{2} v   \) , f   \(   x- \frac{\ep}{2} v  \) ; \frac{1}{2} \)
\geq 0 .
\]

(3) follows from (1) and (2).
\qed
\end{proof}

There exists a discontinuous strictly quasi-concave function. 

\begin{lem}\label{str_quasi-conc}
Let $k$ be a positive function defined on $[0,+\infty)$. 
Put $k^\circ = k ( \vert \cdot \vert )$.
Suppose that $k$ is strictly decreasing on $[0,+\infty)$.
Then, the following statements hold:
\begin{enumerate}[{\rm (1)}]
\item $k$ is strictly quasi-concave on $[0,+\infty)$.

\item $k^\circ$ is strictly quasi-concave on $\R^n$.
\end{enumerate}
\end{lem}

\begin{proof}
(1) Let $r_0, r_1 \in [0 ,+\infty)$, and $\la \in (0,1)$.
Suppose $r_0 < r_1$.
Then, we have $r_0 < r_\la < r_1$.
Since $k$ is strictly decreasing, we have $k ( r_\la ) > k ( r_1 ) = M_{-\infty} ( k(r_0) , k (r_1 ) ; \la )$.

(2) Let $x_0 ,x_1 \in \R^n$, and $\la \in (0,1)$.
Suppose $x_0 \neq x_1$.
Put $r_0 = \vert x_0 \vert$ and $r_1 = \vert x_1 \vert$.
Since $k$ is strictly decreasing and $\vert x_\la \vert \leq r_\la$, we have $k^\circ ( x_\la ) = k ( \lvert x_\la \rvert )  \geq k ( r_\la )$.
Equality holds if and only if there exists a positive $s$ such that $x_0 =s x_1$. 
Thus, the strict quasi-concavity of $k$ shown in (1) completes the proof.
\qed
\end{proof}

At the end of this subsection, let us review the precise statement of the BBL-inequality with our notation.
See also \cite[Theorem 3.1]{B75}, \cite[Theorem 3.1]{DU} and \cite[Theorem 10.1]{G2002}.

\begin{theorem}{\rm {\cite[Theorem 3.3]{BL}}}\label{BL}
Let $f_0$ and $f_1$ be non-negative integrable functions defined on $\R^n$.
Suppose that the $L^1$-norms of $f_0$ and $f_1$ are both positive.
Let $\ell \in [ -1/n , + \infty ]$, and 
\[
S(y) = \esssup \left. \left\{ M_\ell \(  f_0  \( y_0 \) , f_1 \( y_1 \) ;\la \)  \,   \rvert   \,    \( y_0 , y_1 \) \in \R^n \times \R^n ,\ y_\la =y \right\} ,\
y \in \R^n .
\]
Then, we have
\[
\int_{\R^n} S(y)  \, \d y
\geq M_{\ell / ( 1+n \ell )} \(  \int_{\R^n} f_0 (y) \, \d y  , \int_{\R^n} f_1 (y) \, \d y  ; \la \) .
\]
\end{theorem}

\subsection{Parabolically convex sets}

The notion of {\it $\al$-parabolic convexity} of a subset of $\R^n \times (0,+\infty)$ was introduced in \cite{IS11}.
It is an extension of the usual parabolic convexity introduced in \cite{B96}.
We show some basic properties of parabolically convex sets.

\begin{defi}[{\rm \cite[Definition 3.5]{IS11}}]
Let $E$ be a subset of $\R^n \times (0,+\infty)$, and $\al \in \R$.
$E$ is said to be {\it $\al$-parabolically convex} if, for any $(x_0, t_0 ) , (x_1, t_1) \in E$ and $\la \in [0,1]$, $( x_\la , M_\al ( t_0 , t_1 ; \la ) ) \in E$ holds.
\end{defi}

We remark that the original parabolic convexity \cite{B96} corresponds to the case where $\al =1/2$.

\begin{ex}
Let $0 \leq a < b \leq + \infty$, $I =(a,b)$, and $\al \in \R$.
Put 
\[
E = 
\begin{cases}
\left\{ (x,t)     \,    \lvert   \,    \vert x \vert < t^\al ,\ t \in I \right\}  \right.    &( \al \neq 0 ) ,\\
\left\{ (x,t)     \,   \lvert   \,    \vert x \vert < \log t ,\ t \in I \right\} \right.    &( \al =0 ) .
\end{cases}
\]
Then, $E$ is $\al$-parabolically convex.
\end{ex}

We show that convex sets in $\R^n$ can generate parabolically convex sets in $\R^n \times (0,+\infty)$.

\begin{prop}\label{conv=>paraconv}
Let $A$ be a convex set in $\R^n$, and $\al \in \R \sm \{ 0 \}$.
Put 
\[
\hat{A}_\al = \left\{ (x,t) \in \R^n \times (0,+\infty )     \,   \lvert    \,    \frac{x}{t^\al} \in A \right\} \right. .
\]
Then, $\hat{A}_\al$ is $\al$-parabolically convex.
\end{prop}

The proof is directly completed by the convex combination
\begin{equation}
\label{combination}
\begin{split}
\frac{x_\la}{M_\al \( t_0 ,t_1 ; \la \)^\al}  
=\frac{(1-\la) t_0^\al}{M_\al \( t_0, t_1 ;\la \)^\al} \frac{x_0}{t_0^\al} +   \frac{\la t_1^\al}{M_\al \( t_0, t_1 ;\la \)^\al}   \frac{x_1}{t_1^\al}   ,\ \\[2mm]
\(   x_0 , t_0  \) ,   \( x_1 , t_1 \) \in \R^n \times (0,+\infty) .
\end{split}
\end{equation}
This is sometimes used throughout this paper.

When we connect the two cases where $\al \neq 0$ and where $\al =0$, we use the following relations:
\begin{align}
\label{log0=1log}
\log M_0 \( t_0 ,t_1; \la \)  &= M_1 \( \log t_ 0 , \log t_1 ; \la \) ,\
\( t_0 ,t_1 , \la \) \in (1,+\infty) \times (1, +\infty) \times [0,1]   ;\\
\label{exp1=0exp}
\exp M_1 (t_0 ,t_1 ; \la ) &= M_0 ( e^{t_0} , e^{t_1} ; \la ) ,\
\( t_0 ,t_1 , \la \) \in (0 , +\infty) \times ( 0 , +\infty) \times [0,1] .
\end{align}

\begin{cor}\label{conv=>0paraconv}
Let $A$ be a convex set in $\R^n$.
Put 
\[
\hat{A}_0 = \left\{ (x,t) \in \R^n \times (1, + \infty )    \,   \lvert    \,    \frac{x}{\log t} \in A \right\} \right.  .
\]
Then, $\hat{A}_0$ is $0$-parabolically convex.
\end{cor}

\begin{proof}
Let $\hat{A}_1$ be as in Proposition \ref{conv=>paraconv} with $\al =1$.
We remark 
\[
\hat{A}_0 = \left\{ (x,t) \in \R^n \times (1, + \infty )   \,   \lvert    \,  (x, \log t) \in \hat{A}_1 \right\} \right. .
\]
Thanks to the relation \eqref{log0=1log}, Proposition \ref{conv=>paraconv} with $\al =1$ completes the proof.
\qed
\end{proof}

$\hat{A}_\al$ in Proposition \ref{conv=>paraconv} or in Corollary \ref{conv=>0paraconv} is concretely given when $A$ is a {\it convex cone}, that is, $A$ additionally has the property that, for any $(x , s) \in A \times ( 0 , +\infty)$, $sx \in A$ holds.

\begin{prop}\label{cone}
Let $A$, $\al$, and $\hat{A}_\al$ be as in Proposition \ref{conv=>paraconv}.
$A$ is a convex cone if and only if $\hat{A}_\al  = A \times (0,+\infty )$.
\end{prop}

\begin{proof}
Suppose that $A$ is a convex cone.
Let $(x,t) \in \hat{A}_\al$. 
By the definition, we have $x/t^\al \in A$.
Since $A$ is a convex cone, we have $x = t^\al ( x / t^\al ) \in A$.
Thus, $(x,t) \in A \times (0,+\infty)$.
On the other hand, let $(x,t) \in A \times (0,+\infty)$.
Since $A$ is a convex cone, we have $x / t^\al  \in A$, that is, $(x,t) \in \hat{A}_\al$.

Suppose $\hat{A}_\al = A \times (0, +\infty)$.
Let $(x,s) \in A \times (0,+\infty)$. 
Since $(x,s^{-1/\al}) \in A \times (0,+\infty)  = \hat{A}_\al$, we obtain $s x = x / ( s^{-1/\al} )^\al \in A$.
\qed
\end{proof}

\begin{rem}\label{A0A1}
Let $A$ and $\hat{A}_0$ be as in Corollary \ref{conv=>0paraconv}.
Let $\hat{A}_1$ be as in Proposition \ref{conv=>paraconv} with $\al =1$.
Then, the following statements hold:
\begin{enumerate}[(1)]
\item $\hat{A}_0 = \{ (x,t) \in \R^n \times (1 , +\infty) \,  \vert   \, ( x, \log t) \in \hat{A}_1 \}$ (which was mentioned in the proof of Corollary \ref{conv=>0paraconv}).

\item $\hat{A}_1 = \{ (x,t) \in \R^n \times (0 , +\infty) \,  \vert   \, ( x, e^t ) \in \hat{A}_0 \}$.

\item $\hat{A}_0 = A \times (1,+\infty)$ if and only if $\hat{A}_1 = A \times (0,+\infty)$.
\end{enumerate}
\end{rem}

\begin{cor}\label{cone0}
Let $A$ and $\hat{A}_0$ be as in Corollary \ref{conv=>0paraconv}.
A is a convex cone if and only if $\hat{A}_0 = A \times (1, + \infty )$.
\end{cor}

\begin{proof}
Thanks to Remark \ref{A0A1} (3), Proposition \ref{cone} with $\al =1$ completes the proof.
\qed
\end{proof}

Conversely, parabolically convex sets in $\R^n \times (0,+\infty)$ naturally generate convex sets in $\R^n$ since $t  = M_\al ( t ,t    ;   \la )$.

\begin{rem}\label{paraconv=>conv}
Let $E$ be a subset of $\R^n \times (0,+\infty)$, and $\al \in \R$.
Put 
\[
\check{E}(t) = \left. \left\{ x \in \R^n   \,   \rvert   \,   (x,t) \in E \right\} , \
t \in (0,+\infty ).
\] 
If $E$ is $\al$-parabolically convex, then, for each $t \in (0,+\infty )$, $\check{E}(t)$ is convex (or empty).
\end{rem}

For each $\al \in \R$, $\al$-parabolically convex sets have the same basic properties as in \cite[Sections 1 and 2]{B96} (which corresponds to the case where $\al =1/2$).
The properties are not used for the proof of our main theorem, but we show them here, which might be of help in understanding the shape of an $\al$-parabolically convex set.
The proofs are slightly different from \cite{B96}.

\begin{rem}\label{B96Intro}
Let $E$ be a subset of $\R^n \times (0,+\infty)$, and $\al \in \R$.
$E$ is $\al$-parabolically convex if and only if, for any $(x_0, t_0 ) , (x_1, t_1) \in E$, both of the following two hold:
\begin{enumerate}[(i)]
\item If $t_0 \neq t_1$, then, for any $\theta \in [0,1]$, 
\[
E \ni 
\begin{cases}
\ds \( \frac{t_1^\al - t_\theta^\al}{t_1^\al -t_0^\al} x_0 + \frac{t_\theta^\al - t_0^\al}{t_1^\al -t_0^\al} x_1 , t_\theta \) &(\al \neq 0 ) ,  \\
\ds \( \frac{\log t_1 -\log t_\theta}{\log t_1 - \log t_0} x_0 + \frac{\log t_\theta -\log t_0}{\log t_1 - \log t_0} x_1 , t_\theta \) & (\al =0 ) .
\end{cases}
\]

\item If $t_0 = t_1$, then, for any $\theta \in [0,1]$, $(x_\theta , t_0) \in E$.
\end{enumerate}
\end{rem}

\begin{prop}\label{B96Thm2.1}
Let $E$ be a subset of $\R^n \times (0,+\infty)$, and $\al \in \R \sm \{ 0 \}$.
Put 
\begin{align*}
\ds \cE_\al (s ; E ) = \left\{  y    \,  \lvert    \,     \( \frac{y}{s} , s^{-1/\al} \) \in  E \right\} \right.  ,\
s \in (0,+\infty ) ;  \\
\ds \ome_\al (x,t) =  \( \frac{x}{t^\al} , \frac{1}{t^\al} \) ,\ (x,t) \in \R^n \times (0,+\infty) .
\end{align*}
The following statements are equivalent:
\begin{enumerate}[{\rm (i)}]
\item $E$ is $\al$-parabolically convex.

\item For any $s_0 ,s_1 \in (0,+\infty )$ and $\theta \in [0,1]$, $(1-\theta ) \cE_\al ( s_0 ;E ) + \theta \cE_\al ( s_1 ; E ) \subset \cE_\al ( s_\theta ; E )$ holds.

\item $\ome_\al (E)$ is convex. 
\end{enumerate}
\end{prop}

\begin{proof}
(ii)$\iff$(iii) follows from 
\[
\left. \ome_\al (E) = 
\left\{ (y,s)  \in \R^n \times (0,+\infty)  \, \rvert \,  y \in \cE_\al (s ; E ) \right\}  .
\]

(i)$\implies$(ii): Let $y_0 \in \cE_\al (s_0 ; E )$, $y_1 \in \cE_\al (s_1 ;E )$, and $\theta \in [0,1]$.
Put $\la = \theta s_1 / s_\theta \in [0,1]$.
Since we have $(y_0 /s_0 , s_0^{-1/\al} ) \in E$ and $(y_1 /s_1 , s_1^{-1/\al} ) \in E$, we obtain
\[
\( \frac{y_\theta}{s_\theta} , s_\theta^{-1/\al} \) 
= \( (1-\la) \frac{y_0}{s_0} + \la \frac{y_1}{s_1} , M_\al \( s_0^{-1/\al} , s_1^{-1/\al} ; \la \) \) 
\in E  .
\]

(ii)$\implies$(i): Let $(x_0 ,t_0 ) , (x_1 ,t_1 ) \in E$, and $\la \in [0,1]$.
Put $\theta = \la t_1^\al / M_\al ( t_0 , t_1 ; \la )^\al \in [0,1]$.
Since we have $t_0^{-\al} x_0  \in \cE_\al ( t_0^{-\al} ;E )$ and $t_1^{-\al} x_1  \in \cE_\al ( t_1^{-\al} ;E)$, we have $(1-\theta) t_0^{-\al} x_0 + \theta t_1^{- \al} x_1   \in  \cE_\al ( (1-\theta ) t_0^{-\al} + \theta t_1^{-\al} ;E )$.
Hence, we obtain
\[
\( x_\la , M_\al \( t_0 ,t_1 ;\la \) \)
= \( \frac{(1-\theta) t_0^{-\al} x_0 + \theta t_1^{-\al} x_1}{(1-\theta ) t_0^{-\al} + \theta t_1^{-\al}}, \( (1-\theta ) t_0^{-\al} + \theta t_1^{-\al} \)^{-1/\al} \) \in E   .    
\]
\qed
\end{proof}

\begin{rem}\label{0conv=1conv}
Let $E$ be a subset of $\R^n \times (1,+\infty)$, and $\tilde{E} =\{ (x, \log t) \,  \vert   \,  (x,t) \in E \}$.
Put 
\begin{align*}
\cE_0 (s; E ) = \left\{ y    \,   \lvert     \,   \( \frac{y}{s} , e^{1/s} \) \in E \right\}  ,\right.  
s \in (0,+\infty ) ;\\
\ome_0 (x,t) = \( \frac{x}{\log t} , \frac{1}{\log t} \) ,\ 
(x,t) \in \R^n \times (1,+\infty) .
\end{align*}
Let $\cE_1$ and $\ome_1$ be as in Proposition \ref{B96Thm2.1} with $\al =1$.
Then, the following statements hold:
\begin{enumerate}[{\rm (1)}]
\item $E$ is $0$-parabolically convex if and only if $\tilde{E}$ is $1$-parabolically convex.

\item $\cE_0 ( s ; E) = \cE_1 ( s ; \tilde{E} )$ for any $s \in (0,+\infty)$.

\item $\ome_0 (E) = \ome_1 ( \tilde{E} )$.
\end{enumerate}
\end{rem}

\begin{cor}\label{B96Thm2.10}
Let $E$, $\cE_0$ and $\ome_0$ be as in Remark \ref{0conv=1conv}.
The following statements are equivalent:
\begin{enumerate}[{\rm (i)}]
\item $E$ is $0$-parabolically convex.

\item For any $s_0 ,s_1 \in (0,+\infty )$ and $\theta \in [0,1]$, $(1-\theta ) \cE_0 ( s_0 ; E) + \theta \cE_0 ( s_1 ;E ) \subset \cE_0 ( s_\theta ;E )$ holds.

\item $\ome_0 (E)$ is convex. 
\end{enumerate}
\end{cor}

\begin{proof}
Thanks to Remark \ref{0conv=1conv}, Proposition \ref{B96Thm2.1} with $\al =1$ completes the proof.
\qed
\end{proof}

The set $\tilde{E}$ in Remark \ref{0conv=1conv} is concretely given when $E$ is a convex cylinder.

\begin{rem}\label{tildeAI}
Let $A$ be a convex subset of $\R^n$, and $I$ an interval in $(1,+\infty)$.
Let $\tilde{{\color{white}E}}$ be the operator as in Remark \ref{0conv=1conv}.
Then, $\tilde{A \times I} = A \times \log I$.
\end{rem}

\begin{prop}\label{B96Cor2.1}
Let $A$ be a subset of $\R^n$, $I$ an interval in $(0,+\infty)$, and $\al \in \R$.
$A \times I$ is $\al$-parabolically convex if and only if $A$ is convex.
\end{prop}

\begin{proof}
The ``only if'' part follows from Remark \ref{paraconv=>conv}.
The ``if'' part follows from definition.
\qed
\end{proof}

\subsection{Parabolically power concave functions}

The notion of {\it parabolic power concavity} of a function was introduced in \cite{IS14Ann} (see also \cite{IS11}).
In this subsection, we slightly extend the notion and show several basic properties of parabolically power concave functions.

\begin{defi}\label{def_alpconc}
Let $\al \in \R$, $E$ an $\al$-parabolically convex set in $\R^n \times (0,+\infty )$, $\f$ a non-negative function defined on $E$, and $p \in \R \cup \{ \pm \infty \}$.
$\f$ is said to be {\it $\al$-parabolically $p$-concave} on $E$ if, for any $(x_0 ,t_0 ) , ( x_1 ,t_1) \in E$ and $\la \in [0,1]$, the inequality 
\begin{equation}
\label{ineq-alpconc}
\f \( x_\la  , M_\al \( t_0 , t_1 ; \la \) \) \geq M_p \( \f \( x_0 ,t_0  \) , \f \( x_1 ,t_1  \) ; \la \)
\end{equation}
holds.
$\f$ is said to be {\it strictly $\al$-parabolically $p$-concave} on $E$ if both of the following conditions hold:
\begin{enumerate}[(i)]
\item $\f$ is $\al$-parabolically $p$-concave on E.

\item Equality in \eqref{ineq-alpconc} holds if and only if any of the conditions $(x_0, t_0) = ( x_1 ,t_0)$, $\la =0$ or $\la =1$ holds.
\end{enumerate}
When $\al \neq 0$, $\f$ is said to be {\it almost-strictly $\al$-parabolically $p$-concave} on $E$ if both of the following conditions hold:
\begin{enumerate}[(i)]
\item $\f$ is $\al$-parabolically $p$-concave on E.

\item Equality in \eqref{ineq-alpconc} holds if and only if any of the conditions $x_0 /t_0^\al = x_1 / t_1^\al$, $\la =0$ or $\la =1$ holds.
\end{enumerate}
When $E \subset \R^n \times (1, + \infty )$, $\f$ is said to be {\it almost-strictly $0$-parabolically $p$-concave} on $E$ if both of the following conditions hold:
\begin{enumerate}[(i)]
\item $\f$ is $\al$-parabolically $p$-concave on E.

\item Equality in \eqref{ineq-alpconc} holds if and only if any of the conditions $x_0/ \log t_0 = x_1 / \log t_1$, $\la =0$ or $\la =1$ holds.
\end{enumerate}
\end{defi}

Similarly to Definition \ref{def_pconc}, $\al$-parabolic $0$-concavity and $\al$-parabolic $-\infty$-concavity are also called $\al$-parabolic log-concavity and $\al$-parabolic quasi-concavity, respectively.

For $\al \in \R$ and $p \in \R \cup \{ \pm \infty \}$, the composition of an $\al$-parabolically $p$-concave function with a homothety is $\al$-parabolically $p$-concave.
To be precise:

\begin{rem}\label{homo-alpconc}
Let $\al$, $E$, $\f$ and $p$ be as in Definition \ref{def_alpconc}.
Let $s \in \R \sm \{ 0 \}$, and $\tau \in (0,+\infty)$.
Put $E_{s, \tau} = \{ (x,t ) \in \R^n \times (0,+\infty )  \, \vert \, (s x  , \tau t ) \in E \}$.
Then, the following statements hold:
\begin{enumerate}[(1)]
\item If $\f$ is $\al$-parabolically $p$-concave (resp. strictly $\al$-parabolically $p$-concave) on $E$, then the function $E_{s, \tau} \ni (x,t) \mapsto \f ( sx ,\tau t)$ is $\al$-parabolically $p$-concave (resp. strictly $\al$-parabolically $p$-concave) on $E_{s, \tau}$.

\item If $\al \neq 0$ and $\f$ is almost-strictly $\al$-parabolically $p$-concave on $E$, then the function $E_{s, \tau} \ni (x,t) \mapsto \f ( sx ,\tau t)$ is almost-strictly $\al$-parabolically $p$-concave on $E_{s, \tau}$.
\end{enumerate}
\end{rem}

As we see in the next proposition, $0$-parabolically $p$-concave functions defined on a $0$-parabolically convex set generate $1$-parabolically $p$-concave functions defined on a $1$-parabolically convex set, and vice versa.

\begin{prop}\label{0pconc=1pconc}
Let $E$ be a $0$-parabolically convex subset of $\R^n \times (1, +\infty)$, $\f$ a non-negative function defined on $E$, and $p \in \R \cup \{ \pm \infty \}$.
Let $\tilde{E}$ be as in Remark \ref{0conv=1conv}.
Put 
\[
\tilde{\f} (x,t) = \f \( x, e^t \) ,\ (x,t) \in \tilde{E} .
\]
Then, $\f$ is $0$-parabolically $p$-concave (resp. almost-strictly/strictly $0$-parabolically $p$-concave) on $E$ if and only if $\tilde{\f}$ is $1$-parabolically $p$-concave (resp. almost-strictly/strictly $1$-parabolically $p$-concave) on $\tilde{E}$.
\end{prop}

\begin{proof}
By Remark \ref{0conv=1conv} (1), $\tilde{E}$ is $1$-parabolically convex.
The relations \eqref{log0=1log} and \eqref{exp1=0exp} complete the proof.
\qed
\end{proof}

We show that $p$-concave functions can generate $\al$-parabolically $p$-concave functions.

\begin{prop}\label{pconc=>alpconc}
Let $A$ be a convex set in $\R^n$, $f$ a non-negative function defined on $A$, $\al \in \R \sm \{ 0 \}$, and $p \in \R \cup \{ \pm \infty \}$. 
Let $\hat{A}_\al$ be as in Proposition \ref{conv=>paraconv}.
Put
\[
\hat{f}_{p, \al} (x,t) = 
\begin{cases}
\ds t^{\al/p} f \( \frac{x}{t^\al} \) & (p \neq 0 ) ,\\
\ds \exp \( t^\al \log f \( \frac{x}{t^\al}  \) \) &( p =0 ) ,
\end{cases}
\ (x,t) \in \hat{A}_\al .
\]
If $f$ is $p$-concave (resp. strictly $p$-concave) on $A$, then $\hat{f}_{p, \al}$ is $\al$-parabolically $p$-concave (resp. almost-strictly $\al$-paraboclially $p$-concave) on $\hat{A}_\al$.
\end{prop}

\begin{proof}
We give a proof for the case where $p \neq 0$.
The argument in the case where $p =0$ goes parallel.

Let $(x_0 ,t_0) , (x_1, t_1) \in \hat{A}_\al$, and $\la \in [0,1]$.
Using the convex combination \eqref{combination}, the $p$-concavity of $f$ implies
\begin{align*}
\hat{f}_{p, \al}  \( x_\la , M_\al \( t_0 ,t_1 ;\la \) \) 
&= M_\al \( t_0, t_1 ; \la \)^{\al / p}  f    \( \frac{(1-\la) t_0^\al}{M_\al \( t_0, t_1 ;\la \)^\al}  \frac{x_0}{t_0^\al}  + \frac{\la t_1^\al}{M_\al \( t_0, t_1 ;\la \)^\al}  \frac{x_1}{t_1^\al}  \)  \\
&\geq   M_\al \( t_0, t_1 ; \la \)^{\al / p}    M_p \( f   \(    \frac{x_0}{t_0^\al}    \)  , f \(    \frac{x_1}{t_1^\al}    \) ; \frac{\la t_1^\al}{M_\al \( t_0, t_1 ;\la \)^\al} \)   \\
&= M_p \( \hat{f}_{p, \al}  \( x_0, t_0 \) , \hat{f}_{p, \al}  \( x_1, t_1 \) ; \la \) .
\end{align*}
\qed
\end{proof}

\begin{cor}
Let $A$, $f$ and $p$ be as in Proposition \ref{pconc=>alpconc}.
Let $\hat{f}_{p, 1}$ be as in Proposition \ref{pconc=>alpconc} with $\al =1$.
Let $\hat{A}_0$ be as in Corollary \ref{conv=>0paraconv}.
Put 
\[
\hat{f}_{p, 0} (x, t ) = \hat{f}_{p, 1}  \( x, \log t \) ,\
(x,t) \in \hat{A}_0 .
\]
If $f$ is $p$-concave (resp. strictly $p$-concave)on $A$, then $\hat{f}_{p, 0}$ is $0$-parabolically $p$-concave (resp. almost-strictly $0$-parabolically $p$-concave) on $\hat{A}_0$.
\end{cor}

\begin{proof}
Propositions \ref{pconc=>alpconc} with $\al =1$ and \ref{0pconc=1pconc} complete the proof (see also Remark \ref{A0A1}).
\qed
\end{proof}

We show that Proposition \ref{pconc=>alpconc} constructs radially symmetric parabolically power concave functions.

\begin{prop}\label{rsymm}
Let $\kappa$ be a non-negative function defined on $[0,+\infty ) \times (0,+\infty)$, $\al \in \R \sm \{ 0 \}$, $p \in \R  \cup \{ \pm \infty \}$, and $\tau \in (0,+\infty)$.
We consider the following conditions for $\kappa$:
\begin{enumerate}[{\rm (i)}]
\item For any $(r,t) \in [0,+\infty ) \times (0,+\infty )$, we have
\[
\kappa  (r,t) =
\begin{cases}
\ds t^{\al /p} \kappa  \(  \frac{r}{t^\al} , \tau \)  &( p \neq 0 ) , \\
\ds  \exp \(  t^\al  \log   \kappa  \(  \frac{r}{t^\al}  , \tau \) \)   &(p=0) .
\end{cases}
\]

\item $\kappa ( \cdot , \tau )$ is $p$-concave (resp. strictly $p$-concave) on $[0,+\infty )$.

\item For each $t \in (0,+\infty )$, $\kappa (\cdot ,t )$ is decreasing (resp. strictly decreasing) on $[0,+\infty )$.
\end{enumerate}
Put
\[
\kappa^\circ   (x,t) = \kappa    \( \vert x \vert , t \)  ,\ (x,t) \in \R^n \times (0,+\infty) .
\]
Then, the following statements hold:
\begin{enumerate}[{\rm (1)}]
\item If {\rm (i)} and {\rm (ii)} are satisfied, then $\kappa$ is $\al$-parabolically $p$-concave (resp. almost-strictly $\al$-parabolically $p$-concave) on $[0,+\infty ) \times (0,+\infty)$.

\item If {\rm (i)}, {\rm (ii)} and {\rm (iii)} are satisfied, then $\kappa^\circ$ is $\al$-parabolically $p$-concave (resp. almost-strictly $\al$-parabolically $p$-concave) on $\R^n \times (0,+\infty)$.
\end{enumerate}
\end{prop}

\begin{proof}
(1) In Proposition \ref{pconc=>alpconc}, put $n=1$, $A=[0, +\infty)$, and $f = \kappa (\cdot ,\tau)$.
Since $A$ is a convex cone in $\R$, by Proposition \ref{cone}, we have $\hat{A}_\al = A \times (0,+\infty)$.
Thus, Proposition \ref{pconc=>alpconc} completes the proof.

(2) Let $(x_0,t_0), (x_1, t_1) \in \R^n \times (0,+\infty)$, and $\la \in [0,1]$.
Put $r_0 = \vert x_0 \vert$ and $r_1 = \vert x_1 \vert$. 
By the condition (iii) and $\vert x_\la \vert \leq r_\la$, we have 
\[
\kappa^\circ \( x_\la , M_\al \( t_0 ,t_1 ; \la \) \) =  \kappa  \( \lvert x_\la \rvert  , M_\al \( t_0 ,t_1 ; \la \)  \) 
\geq  \kappa \( r_\la , M_\al \( t_0 ,t_1 ; \la \)  \) .
\]
When the condition (iii) is strictly satisfied, equality holds if and only if there exists a positive $s$ such that $x_0 = s x_1$.
Thus, the $\al$-parabolic $p$-concavity of $\kappa$ shown in (1) completes the proof.
\qed
\end{proof}

\begin{cor}\label{0para_rsymm}
Let $\kappa$, $p$, $\tau$ and $\circ$ be as in Proposition \ref{rsymm}.
Put 
\[
\kappa_0 (r,t) = \kappa ( r , \log t) ,\ (r,t) \in [0,+\infty ) \times (1,+\infty) .
\]
Then, the following statements hold:
\begin{enumerate}[{\rm (1)}]
\item If {\rm (i)} and {\rm (ii)} in Proposition \ref{rsymm} with $\al =1$ are satisfied, then $\kappa_0$ is $0$-parabolically $p$-concave (resp. almost-strictly $0$-parabolically $p$-concave) on $[0,+\infty ) \times (1 , +\infty)$.

\item If {\rm (i)}, {\rm (ii)} and {\rm (iii)} in Proposition \ref{rsymm} with $\al =1$ are satisfied, then $\kappa_0^\circ$ is $0$-parabolically $p$-concave (resp. almost-strictly $0$-parabolically $p$-concave) on $\R^n \times (1 , +\infty)$.
\end{enumerate}
\end{cor}

\begin{proof}
Propositions \ref{rsymm} with $\al =1$ and \ref{0pconc=1pconc} complete the proof (see also Remarks \ref{A0A1} and \ref{tildeAI}).
\qed
\end{proof}

Conversely, $\al$-parabolically $p$-concave functions naturally generate $p$-concave functions since $\tau = M_\al ( \tau , \tau ; \la )$.

\begin{rem}\label{alpconc=>pconc}
Let $\al \in \R$, $E$ an $\al$-parabolically convex set in $\R^n \times (0,+\infty )$, $\f$ a non-negative function defined on $E$, $p \in \R \cup \{ \pm \infty \}$, and $\tau \in (0,+\infty)$.
Let $\check{E}$ be as in Remark \ref{paraconv=>conv}. 
Suppose $\check{E} (\tau) \neq \emptyset$.
Put 
\[
\check{\f}_\tau   (x) = \f (x, \tau ) ,\
x \in \check{E} (\tau) .
\]
If $\f$ is $\al$-parabolically $p$-concave (resp. strictly/almost-strictly $\al$-parabolically $p$-concave) on $E$, then $\check{\f}_\tau$ is $p$-concave (resp. strictly $p$-concave) on $\check{E}   (\tau)$.
\end{rem}

\section{Main theorem and its applications}

\subsection{Lemmas for the main theorem} 

\begin{lem}\label{concavityPhi}
Let $I$ be an interval in $(0,+\infty)$, and $\f$ a non-negative function defined on $\R^n \times I$, $\al \in \R$, and $p \in \R \cup \{ \pm \infty \}$.
Put
\[
\Phi  (x, y , t) = \f ( x-y , t )  ,\
(x, y ,t) \in \R^n \times \R^n \times I  .
\]
If $\f$ is $\al$-parabolically $p$-concave on $\R^n \times I$, then $\Phi$ is $\al$-parabolically $p$-concave on $\R^n \times \R^n \times I$.
\end{lem}

\begin{proof}
Let $(x_0,  y_0 , t_0), (x_1 , y_1 , t_1 ) \in \R^n \times \R^n \times I$, and $\la \in [0,1]$.
Since $\f$ is $\al$-parabolically $p$-concave on $\R^n \times I$, we have
\begin{align*}
\Phi \( x_\la  ,y_\la , M_\al \( t_0 , t_1 ; \la \) \)    
&= \f \( (1-\la ) \( x_0 -y_0 \) + \la \( x_1-y_1\) , M_\al \( t_0 , t_1 ; \la \) \)   \\ 
&\geq M_p \( \f \( x_0 -y_0 , t_0 \) , \f \( x_1 -y_1 , t_1 \) ; \la \)    \\
&= M_p \( \Phi  \( x_0, y_0 ,t_0 \)  , \Phi   \( x_1 , y_1 ,t_1 \) ; \la \) .
\end{align*}
\qed
\end{proof}

\begin{rem}\label{equalityPhi}
Let $I$, $\f$, $\al$, $p$ and $\Phi$ be as in Lemma \ref{concavityPhi}.
Suppose that $\f$ is almost-strictly $\al$-parabolically $p$-concave on $\R^n \times I$.
Then, the following statements hold:
\begin{enumerate}[(1)]
\item We have
\begin{align*}
\Phi  \( x_\la , y_\la ,  M_\al \( t_0 , t_1 ; \la \) \)  
=M_p \( \Phi \( x_0 , y_0 ,t_0  \)  , \Phi   \( x_1 , y_1 ,t_1 \) ; \la \) \\
\iff
\begin{cases}
\ds \frac{x_0 -y_0}{t_0^\al} = \frac{x_1-y_1}{t_1^\al}    & ( \al \neq 0 ) ,\\
\ds \frac{x_0 -y_0}{\log t_0} = \frac{x_1-y_1}{\log t_1}  & ( \al = 0 ) . 
\end{cases}
\end{align*}

\item For each $(x,t) \in \R^n \times I$, $\Phi (x ,  \cdot  ,t)$ is strictly $p$-concave on $\R^n$ (see also Remarks \ref{composition} and \ref{alpconc=>pconc}).
\end{enumerate}
\end{rem}

\begin{lem}\label{inequalityS}
Let $\Phi_0 , \Phi_1 \in [0, + \infty )$, and $\psi$ a non-negative function defined on $\R^n$.
Let $p$ and $q \in \R \cup \{ \pm \infty \}$.
Let $\ell$ be as in \eqref{Lemma10.1ell}.
Suppose that $\psi$ is $q$-concave on $\R^n$, and that $p +q \geq 0$.
Then, for any $y_0 , y_1  \in  \R^n$ and $\la \in [0,1]$, we have
\[
M_p \( \Phi_0 , \Phi_1 ; \la \) \psi \( y_\la \)
\geq M_\ell \( \Phi_0 \psi \( y_0  \) , \Phi_1 \psi \( y_1  \) ; \la \)  .
\]
\end{lem}

\begin{proof}
\eqref{Lemma10.1} with $(a, b , c ,d ) =( \Phi_0 ,  \Phi_1 ,  \psi (y_0) ,  \psi (y_1))$ completes the proof.
\qed
\end{proof}

\begin{lem}\label{clint}
Let $\Ome$ be a convex set in $\R^n$ with non-empty interior, and $x \in \cl \Ome$.
Then, $x \in \cl \interi \Ome$.
\end{lem}

This is a consequence of \cite[Lemma 1.1.9]{Sch} (see also \cite[Exercise 3.8]{Mosz}).

\begin{prop}\label{geometric_lemma}
Let $\Ome$ be a bounded convex set in $\R^n$ with non-empty interior, $K = \cl \Ome$, $s \in (0,1]$, $\mu \in [0,+\infty )$, and $v \in S^{n-1}$.
Suppose $(s, \mu ) \neq (1,0)$. 
Then, $\Ome \sm (s K - \mu v)$ has an interior point.
\end{prop}

\begin{proof}
Let us first show the statement under the assumption $h_K(v) >h_{sK-\mu v} (v)$.
Since we have $s K - \mu v \subset H^- (h_{sK-\mu v} (v) , v)$, it is sufficient to show that $\Ome \sm H^- (h_{sK-\mu v} (v) ,  v )$ has an interior point.

Let $x \in K$ be such that $x\cdot v =h_K (v)$.
By Lemma \ref{clint}, we have
\[
B \( x , \frac{h_K (v) - h_{sK-\mu v} (v)}{2} \) \cap \interi \Ome \neq \emptyset .
\]
We take a point $y$ from the above intersection.
We remark
\[
y \cdot v
= x \cdot v + (y-x) \cdot v 
\geq h_K (v) - \vert y -x \vert 
> h_K (v) -  \frac{h_K (v) - h_{sK-\mu v} (v)}{2} 
> h_{sK-\mu v} (v)  ,
\]
that is, $y \notin H^- ( h_{sK-\mu v} (v)  , v )$. 
Let us show that $y$ is an interior point of $\Ome \sm H^- (  h_{sK-\mu v} (v)  , v)$.

Since $y \in \interi \Ome$, there exists a positive $\de$ such that $B(y,\de) \subset \Ome$. 
Let 
\[
\ep = \min \left\{ \frac{h_K (v) - h_{sK-\mu v} (v)}{2}  , \de \right\} >0 .
\]
Fix an arbitrary $z \in B(y,\ep)$. 
By the definition of $\ep$, we have $z \in \Ome$.
Since we have
\[
\vert z -x \vert 
\leq \vert z-y \vert + \vert y-x \vert 
< \ep + \frac{h_K (v) - h_{sK-\mu v} (v)}{2} 
\leq h_K (v) - h_{sK-\mu v} (v) ,
\]
we have 
\[
z \cdot v
= x\cdot v + (z-x) \cdot v
\geq h_K (v) - \vert z-x \vert 
> h_{sK-\mu v} (v)  .
\]
Thus, $z \notin H^- ( h_{sK-\mu v} (v)  , v)$. 

Next, we show the statement under the assumption $h_K(v) \leq h_{sK-\mu v} (v)$.
Since $h_{sK-\mu v} (v) = s h_K(v) -\mu$, the assumption implies $\mu \leq (s-1) h_K (v)$ and $s<1$.
Since we have $s K - \mu v \subset H^- ( h_{sK -\mu v} (-v) , -v)$, it is sufficient to show that $\Ome \sm H^- (h_{sK -\mu v} (-v) ,  -v )$ has an interior point.

Since $\Ome$ has an interior point, we have the positivity of the width of $K$, that is, $h_K (v) + h_K (-v) >0$.
Thus, we obtain
\[
h_K (-v) - h_{sK -\mu v} (-v) 
=h_K (-v) - \( sh_K (-v) +\mu \) 
\geq (1-s)  \( h_K (v) + h_K (-v) \) 
> 0 .
\]

Let $x \in K$ be such that $h_K (-v) = x \cdot (-v)$. 
By Lemma \ref{clint}, we have 
\[
B \( x , \frac{h_K (-v) - h_{sK -\mu v} (-v)}{2} \) \cap \interi \Ome \neq \emptyset .
\]
We take a point $y$ from the above intersection.
In the same manner as above, it is shown that $y$ is an interior point of $\Ome \sm (sK - \mu v)$. 
\qed
\end{proof}

\subsection{Main theorem}

\begin{theorem}\label{thm_alpconc}
Let $I$ be an interval in $(0,+\infty)$, $\f$ a non-negative measurable function defined on $\R^n \times I$, $\psi$ a non-negative measurable function defined on $\R^n$, $\al \in \R$, $p \in \R$, and $q \in \R \cup \{ + \infty \}$.
Let $\ell$ be as in \eqref{Lemma10.1ell}.
Assume that the following conditions are satisfied:
\begin{enumerate}[{\rm (i)}]
\item $\f$ is almost-strictly $\al$-parabolically $p$-concave on $\R^n \times I$.

\item $\psi$ is $q$-concave on $\R^n$.

\item $\R^n \sm \psi^{-1} (0)$ is bounded, and its interior is not empty.

\item $p+q \geq 0$ and $\ell \geq -1/n$. 
\end{enumerate}
Then, the function 
\[
\Ga (x,t) 
= \f ( \cdot ,t ) \ast \psi (x) 
= \int_{\R^n} \f (x-y , t) \psi (y) \, \d y ,\  (x,t) \in \R^n \times I ,
\]
is strictly $\al$-parabolically $\ell /(1+n \ell )$-concave on $\R^n \times I$. 
\end{theorem}

\begin{lem}\label{lem_0pconc}
If Theorem \ref{thm_alpconc} is true for $\al =1$, then it is true for $\al =0$.
\end{lem}

\begin{proof}
We assume $I \subset (1,+\infty)$ when we discuss almost-strict $0$-parabolic power concavity of a function (see Definition \ref{def_alpconc}).
By Remark \ref{tildeAI}, $\tilde{\R^n \times I} = \R^n \times \log I \subset \R^n \times (0,+\infty)$.
Let $\tilde{\f}$ be as in Proposition \ref{0pconc=1pconc}.
By the condition (i) with $\al =0$ and Proposition \ref{0pconc=1pconc}, $\tilde{\f}$ is almost-strictly $1$-parabolically $p$-concave on $\R^n \times \log I$.
Thus, by Theorem \ref{thm_alpconc} with $\al =1$, the function
\[
\tilde{\Ga} (x, t) = \tilde{\f} ( \cdot , t) \ast \psi (x) 
= \int_{\R^n} \tilde{\f}   (x-y , t )  \psi (y) \, \d y  ,\
(x,t) \in \R^n \times \log I ,
\]
is strictly $1$-parabolically $\ell /(1+n \ell )$-concave on $\R^n \times \log I$.
Since $\Ga (x,t) = \tilde{\Ga} ( x, \log t )$ for any $(x,t) \in \R^n \times I$, Proposition \ref{0pconc=1pconc} completes the proof.
\qed
\end{proof}

\begin{proof}{(\rm of Theorem \ref{thm_alpconc})}
Due to Lemma \ref{lem_0pconc}, we give a proof in the case where $\al \neq 0$.

Let $(x_0 ,t_0) , (x_1 , t_1) \in \R^n \times I$, and $\la \in (0,1)$. 
Suppose $(x_0 ,t_0 ) \neq (x_1 ,t_1)$. 
Put
\begin{align*}
\Phi   (x, y ,t ) = \f (x-y ,t)  , \ 
\Phi_\psi (x,y,t) = \Phi   (x,y,t) \psi (y) ,\
(x , y, t ) \in \R^n \times \R^n \times I   ,   \\
S (y) = \left. \esssup \left\{ M_\ell \( \Phi_\psi \(  x_0 ,y_0 , t_0 \)  , \Phi_\psi \( x_1 ,y_1 , t_1  \) ; \la  \)   \,  \rvert   \,    \( y_0, y_1 \) \in \R^n \times \R^n ,\ y_\la =y \right\}  ,\  
y \in \R^n .
\end{align*}
Let $\Ome =\R^n \sm \psi^{-1} (0)$. 
By Remark \ref{equalityPhi} (2) and Lemma \ref{continuity} (2), for each $(x,t) \in \R^n \times I$, $\Phi_{\psi} (x, \cdot ,t )$ is positive on $\interi \Ome$.
By Theorem \ref{BL} with $f_0 = \Phi_\psi (x_0 , \cdot , t_0)$ and $f_1 = \Phi_\psi (x_1 , \cdot , t_1)$,
\[
\int_{\R^n} S(y) \, \d y 
\geq M_{\ell  /(1+n  \ell  )} \( \Ga \( x_0, t_0 \) ,  \Ga \( x_1, t_1 \) ; \la \) .
\]
Thus, it is sufficient to show
\[
\int_{\R^n} S(y) \, \d y 
< \Ga \(   x_\la , M_\al \( t_0 ,t_1 ;\la \) \) 
=\int_{\R^n} \Phi_\psi \( x_\la , y , M_\al \( t_0 , t_1 ; \la \) \) \, \d y .
\]

Since $\Phi_\psi ( x_\la , y, M_\al ( t_0, t_1, \la ) ) =0$ for any $y \in \R^n \sm \Ome$, we have 
\[
\Ga \( x_\la ,  M_\al \(  t_0, t_1, \la  \)  \) 
= \int_\Ome \Phi_\psi   \( x_\la , y, M_\al  \( t_0, t_1, \la  \)  \)  \, \d y  .
\]
Since $\Ome$ is convex (see Remark \ref{supp_convex}), if $y \in \R^n \sm \Ome$, then, for any $(y_0, y_1) \in \R^n \times \R^n$ with $y_\la =y$, we have $(y_0 ,y_1 ) \notin \Ome \times \Ome$.
From this property, we have $S(y)=0$ for any $y \in \R^n \sm \Ome$, which implies
\[
\int_{\R^n} S(y) \, \d y 
= \int_\Ome S(y) \, \d y .
\]
Thus, our aim is to show
\[
\int_\Ome S(y) \, \d y 
< \int_\Ome \Phi_\psi   \( x_\la , y, M_\al  \( t_0, t_1, \la  \)  \)  \, \d y .
\]

We construct a subset $\Ome'$ of $\Ome$ such that $\Ome'$ has non-empty interior, and that $S(y) <  \Phi_\psi ( x_\la , y , M_\al ( t_0 ,t_1 ; \la ) )$ for any $y \in \Ome'$. 
Let $K= \cl \Ome$. 
Since $M_\ell ( \Phi_\psi ( x_0 ,y_0 , t_0 ) ,  \Phi_\psi ( x_1 ,  y_1 , t_1 ) ;\la )   =0$ for any $(y_0, y_1 ) \notin K \times K$, we have 
\[
S(y) = \left. \esssup \left\{ M_\ell \( \Phi_\psi \(  x_0 ,y_0 , t_0 \)  , \Phi_\psi \( x_1 ,y_1 , t_1  \) ; \la  \) \rvert \( y_0, y_1 \)  \in K \times K ,\ y_\la =y \right\}  .
\]
By Lemma \ref{inequalityS} with $\Phi_0 = \Phi ( x_0, y_0 , t_0 )$ and $\Phi_1 = \Phi (x_1 , y_1 , t_1 )$, we have
\[
S(y) \leq \left. \esssup \left\{ M_p \( \Phi \(  x_0 ,y_0 , t_0 \)  , \Phi \( x_1 ,y_1 , t_1  \) ; \la  \)  \,  \rvert   \,   \( y_0, y_1 \)  \in K \times K ,\ y_\la =y \right\} \psi (y) .
\]
By the continuity of $\Phi ( x_j , \cdot ,t_j  )$ (see Lemma \ref{continuity} and Remark \ref{equalityPhi} (2)) and the compactness of $K$, there exists a pair $(\eta_0 ,\eta_1 ) \in \R^n \times \R^n$ such that $(\eta_0 , \eta_1 ) \in K \times K$, $\eta_\la = y$, and 
\begin{align*}
&\left. \esssup \left\{ M_p \(   \Phi   \(  x_0 ,y_0 , t_0 \)  ,   \Phi   \( x_1 ,y_1 , t_1  \) ; \la  \) \,  \rvert   \,   \( y_0, y_1 \)  \in K \times K ,\ y_\la =y \right\}     \\
&= M_p \( \Phi   \( x_0 ,  \eta_0 , t_0 \) ,   \Phi   \( x_1, \eta_1 , t_1 \) ; \la \) .
\end{align*}

Let 
\[
K' =  t_1^\al \( \( \frac{\la}{t_0^\al} + \frac{1-\la}{t_1^\al} \) K - \la \( \frac{x_0}{t_0^\al} - \frac{x_1}{t_1^\al} \) \)   \cap  \,  t_0^\al \( \( \frac{\la}{t_0^\al} + \frac{1-\la}{t_1^\al} \) K + (1- \la ) \( \frac{x_0}{t_0^\al} - \frac{x_1}{t_1^\al} \) \)   ,
\]
and $\Ome'  = \Ome \sm K'$.
Proposition \ref{geometric_lemma} guarantees that $\Ome'$ has non-empty interior.

It is directly shown that $y \in K'$ if and only if there exists a pair $( y_0 , y_1 ) \in \R^n \times \R^n$ such that 
\begin{numcases}
{}
\( y _0 ,y_1 \) \in K \times K ; \label{KK}   \\
y_\la =y  ;  \label{y}  \\
\frac{x_0-y_0 }{t_0^\al} = \frac{x_1-y_1}{t_1^\al}   .  \label{equality}  
\end{numcases}
If $y \in \Ome'$ is expressed by \eqref{KK} and \eqref{y} for $(y_0 , y_1 ) = ( \eta_1 , \eta_2)$, then \eqref{equality} does not hold for $(y_0 , y_1 ) = ( \eta_1 , \eta_2)$.
Thus, by Remark \ref{equalityPhi} (1), we have
\[
M_p \(   \Phi  \( x_0 ,  \eta_0 , t_0 \) , \Phi \( x_1, \eta_1 , t_1 \) ; \la \)    
< \Phi   \( x_\la ,  \eta_\la , M_\al \( t_0 , t_1 ; \la \) \)   
=  \Phi   \( x_\la ,  y , M_\al \( t_0 , t_1 ; \la \) \)   
\]
for any $y \in  \Ome'$.
Hence we obtain
\[
S(y) < \Phi   \( x_\la ,  y , M_\al \( t_0 , t_1 ; \la \) \) \psi (y)  
= \Phi_\psi \( x_\la , y , M_\al \( t_0 ,t_1; \la \) \) 
\]
for any $y \in \Ome'$, and the proof is completed.
\qed
\end{proof}

\begin{cor}
Let $I$, $\f$, $\psi$, $\al$, $p$, $q$ and $\Ga$ be as in Theorem \ref{thm_alpconc}.
If all the conditions {\rm (i)}--{\rm (iv)} in Theorem \ref{thm_alpconc} are satisfied, then $\Ga$ has at most one maximum point in $\R^n \times I$.
\end{cor}

Theorem \ref{thm_alpconc} improves \cite[Theorem 3.4]{SakJG}.

\begin{theorem}\label{thm_pconc}
Let $f$ and $g$ be non-negative measurable functions defined on $\R^n$.
Let $p \in \R$, and $q \in \R \cup \{ + \infty \}$.
Let $\ell$ be as in \eqref{Lemma10.1ell}.
Assume that the following conditions are satisfied:
\begin{enumerate}[{\rm(i)}]
\item $f$ is strictly $p$-concave on $\R^n$.

\item $g$ is $q$-concave on $\R^n$.

\item $\R^n \sm g^{-1} (0)$ is bounded, and its interior is not empty.

\item $p+q \geq 0$ and $\ell \geq -1/n$. 
\end{enumerate}
Then, the function 
\[
G (x) 
= f  \ast g (x) 
= \int_{\R^n} f (x-y ) g (y) \, \d y ,\  x \in \R^n ,
\]
is strictly $\ell  /  (  1+  n   \ell )$-concave on $\R^n$. 
\end{theorem}

\begin{proof}
Let 
\[
\f (x,t) = t^{1/p} f \( \frac{x}{t} \) ,\ 
(x,t) \in \R^n \times (0,+\infty ) .
\]
Proposition \ref{pconc=>alpconc} guarantees that $\f$ is almost-strictly $1$-parabolically $p$-concave on $\R^n \times (0,+\infty)$ (see also Proposition \ref{cone}).
By Theorem \ref{thm_alpconc}, the function
\[
\Ga (x,t) = \f  (\cdot ,t ) \ast g (x) 
= \int_{\R^n} \f (x-y ,t ) g (y) \, \d y , \
(x,t) \in \R^n \times (0,+\infty ) ,
\]
is strictly 1-parabolically $\ell /(1+n \ell)$-concave on $\R^n \times (0,+\infty)$.
Since $G = \Ga ( \cdot , 1)$, Remark \ref{alpconc=>pconc} completes the proof.
\qed
\end{proof}

\begin{cor}
Let $f$, $g$, $p$, $q$ and $G$ be as in Theorem \ref{thm_pconc}.
If all the conditions {\rm (i)}--{\rm (iv)} in Theorem \ref{thm_pconc} are satisfied, then $G$ has at most one maximum point in $\R^n$.
\end{cor}

\subsection{Applications to concrete convolutions}

In this subsection, we show the strict parabolic power concavity and strict power concavity in space variable of the Gauss--Weierstrass integral \eqref{Weierstrass} and the Poisson integral \eqref{Poisson}. 
As applications of Theorem \ref{thm_alpconc}, the strict $1/2$-parabolic quasi-concavity of the Gauss--Weierstrass integral and the strict $1$-parabolic quasi-concavity of the Poisson integral are given.

\begin{ex}
Let $a \in \R \sm \{ 0 \}$, $b \in [ 1, +\infty )$, and $c \in \R$.
Suppose $c/a <0$.
Put 
\[
\kappa  (r,t) = t^a \exp \( -\frac{r^b}{t^c} \) ,\ 
(r,t) \in [0,+\infty ) \times (0,+\infty ) .
\]
Then, $\kappa$ satisfies the conditions (i)--(iii) in Proposition \ref{rsymm} with $\al =c/b$, $p = c/(ab)$ and $\tau =1$. 
Thus, the function
\[
\kappa^\circ (x,t) =    \kappa   \( \lvert x \rvert , t\) ,\
(x,t) \in \R^n \times (0,+\infty ) ,
\]
is almost-strictly $c/b$-parabolically $c/(ab)$-concave on $\R^n \times (0,+\infty)$.
In particular, applying this investigation with $a=-n/2$, $b=2$ and $c=1$, Remark \ref{homo-alpconc} (2) guarantees that the Gauss--Weierstrass kernel
\[
\R^n \times (0,+\infty )  \ni (x,t) \mapsto \frac{1}{(4 \pi t )^{n /2}} \exp \( -\frac{\vert x \vert^2}{4t} \) 
\]
is almost-strictly $1/2$-parabolically $-1/n$-concave on $\R^n \times (0,+\infty)$. 
\end{ex}

\begin{prop}\label{alpconc_W}
Let $\Ome$ be a bounded convex set in $\R^n$ with non-empty interior.
Let $W$ be as in \eqref{Weierstrass}.
$W \chi_\Ome$ is strictly $1/2$-parabolically quasi-concave on $\R^n \times (0,+\infty )$.
\end{prop}

\begin{ex}
Let $a \in [0,+\infty )$, $b \in [1, +\infty)$, and $c \in (-\infty , 0  )$.
Suppose $(a,b) \neq (0,1)$ and $c < -a$. 
Put 
\[
\kappa   (r,t) = t^a \( r^b + t^b \)^{c/b} ,\
(r,t) \in [0,+\infty ) \times (0,+\infty ) .
\]
Then, $\kappa$ satisfies the conditions (i)--(iii) in Proposition \ref{rsymm} with $\al =1$, $p = 1/(a+c)$ and $\tau =1$. 
Thus, the function
\[
\kappa^\circ (x,t) = \kappa   \( \lvert x \rvert , t    \) ,\
(x,t) \in \R^n \times (0,+\infty ) ,
\]
is almost-strictly $1$-parabolically $1/(a+c)$-concave on $\R^n \times (0,+\infty)$.
In particular, applying this investigation with $a=1$, $b=2$ and $c=-(n+1)$, the Poisson kernel
\[
\R^n \times (0,+\infty )  \ni (x,t) \mapsto \frac{2t}{\sigma_n \( S^n \)} \( \vert x \vert^2 + t^2 \)^{-(n+1)/2} 
\]
is almost-strictly $1$-parabolically $-1/n$-concave on $\R^n \times (0,+\infty)$. 
\end{ex}

\begin{prop}\label{alpconc_P}
Let $\Ome$ be a bounded convex set in $\R^n$ with non-empty interior.
Let $P$ be as in \eqref{Poisson}.
$P \chi_\Ome$ is strictly $1$-parabolically quasi-concave on $\R^n \times (0,+\infty )$.
\end{prop}

As applications of Theorem \ref{thm_pconc},  the strict log-concavity in space variable of the Gauss--Weierstrass integral \eqref{Weierstrass} and the strict power concavity in space variable of the Poisson integral \eqref{Poisson} are given.

\begin{ex}
Let $t \in (0,+\infty)$, $b \in (1, +\infty)$ and $c \in \R$.
Put 
\[
k_t  (r ) =  \exp \( -\frac{r^b}{t^c} \) ,\ 
r \in [0,+\infty ) .
\]
Then, $k_t$ is strictly log-concave on $[0,+\infty )$ and strictly decreasing on $[0,+\infty)$.
Thus, the function 
\[
k_t^\circ ( x ) = k_t \( \vert x \vert \)  ,\ 
x \in \R^n ,
\]
is strictly log-concave on $\R^n$.
In particular, applying this investigation with $b=2$ and $c=1$, the function 
\[
\R^n \ni x  \mapsto \frac{1}{(4 \pi t )^{n /2}} \exp \( -\frac{\vert x \vert^2}{4t} \) 
\]
is strictly log-concave on $\R^n$. 
\end{ex}

\begin{prop}\label{pconc_W}
Let $g$ be a non-negative function defined on $\R^n$, and $q \in \R \cup \{ + \infty \}$.
Assume that the following conditions are satisfied:
\begin{enumerate}[{\rm (i)}]
\item $g$ is $q$-concave on $\R^n$.

\item $\R^n \sm g^{-1} (0)$ is bounded, and its interior is not empty.

\item $q \geq 0$.
\end{enumerate}
Let $W$ be as in \eqref{Weierstrass}.
For any $t \in (0,+\infty)$, the function $Wg (\cdot ,t) : \R^n \to (0,+\infty)$ is strictly log-concave on $\R^n$.
\end{prop}

\begin{ex}
Let $t \in (0,+\infty)$, $b \in ( 1 ,+\infty)$, and $c \in (-\infty , 0)$.
Put 
\[
k_t (r ) = \( r^b + t^b \)^{c/b} ,\ 
r \in [0,+\infty ) .
\]
Then, $k_t$ is strictly $1/c$-concave on $[0,+\infty )$ and strictly decreasing on $[0,+\infty)$.
Thus, the function
\[
k_t^\circ  ( x )  =   k_t \( \vert x \vert \)  ,\
x \in \R^n ,
\]
is strictly $1/c$-concave on $\R^n$.
In particular, applying this investigation with $b=2$ and $c=-(n+1)$, the function
\[
\R^n \ni x \mapsto \frac{2t}{\sigma_n \( S^n \)} \( \vert x \vert^2 + t^2 \)^{-(n+1)/2} 
\]
is strictly $-1/(n+1)$-concave on $\R^n$. 
\end{ex}

\begin{prop}\label{pconc_P}
Let $g$ be a non-negative function defined on $\R^n$, and $q \in \R \cup \{ + \infty \}$.
Assume that the following conditions are satisfied:
\begin{enumerate}[{\rm (i)}]
\item $g$ is $q$-concave on $\R^n$.

\item $\R^n \sm g^{-1} (0)$ is bounded, and its interior is not empty.

\item $q \geq 1$.
\end{enumerate}
Let $P$ be as in \eqref{Poisson}.
For any $t \in (0,+\infty)$, the function $Pg (\cdot ,t ) : \R^n \to (0,+\infty )$ is strictly $q/(1-q)$-concave on $\R^n$.
\end{prop}


\section*{Acknowledgements}

The authors would like to express their gratitude to Professor Paolo Salani for encouraging this research.

The first-named author is partially supported by JSPS Kakenhi (No. 19K03462).
The second-named author is partially supported by JSPS Kakenhi (No. 20K14320) and funds (No.
205004) from the Central Research Institute of Fukuoka University.

\no
Jun O'HARA\\
Department of Mathematics and Informatics, \\
Faculty of Science, \\
Chiba University, \\
1-33 Yayoi-cho, Inage, Chiba, 263-8522, Japan\\
E-mail: ohara@math.s.chiba-u.ac.jp\\

\no
Shigehiro SAKATA\\
Department of Applied Mathematics,\\
Faculty of Science,\\
Fukuoka University,\\
8-19-1 Nanakuma, Jonan, Fukuoka, 814-0180, Japan\\
E-mail: ssakata@fukuoka-u.ac.jp

\end{document}